\documentclass[10pt,letterpaper]{article}
\usepackage[T1]{fontenc}
\usepackage{lmodern}
\usepackage{amsmath,amssymb,amsthm,mathtools}
\usepackage{microtype}
\usepackage{enumitem}
\usepackage[hidelinks]{hyperref}
\usepackage{bookmark}
\input{glyphtounicode}
\hypersetup{
 pdftitle={Existence of a positive hyperbolic orbit in three-dimensional Reeb flows},
 pdfauthor={Taisuke Shibata},
 pdfsubject={Working manuscript: embedded contact homology, Reeb flows and compactness},
 pdfkeywords={Reeb flow, positive hyperbolic orbit, embedded contact homology, U-map}}
\allowdisplaybreaks[2]
\setlist[enumerate]{itemsep=2pt,topsep=4pt}
\newtheorem{theorem}{Theorem}[section]
\newtheorem{lemma}[theorem]{Lemma}
\newtheorem{proposition}[theorem]{Proposition}
\newtheorem{corollary}[theorem]{Corollary}
\theoremstyle{definition}

\theoremstyle{remark}

\newcommand{\R}{\mathbb{R}}
\newcommand{\Z}{\mathbb{Z}}
\newcommand{\F}{\mathbb{F}}

\newcommand{\A}{A}
\newcommand{\Ccal}{\mathcal{C}}
\newcommand{\Mcal}{\mathcal{M}}
\newcommand{\Gcal}{\mathcal{G}}
\newcommand{\Bcal}{\mathcal{B}}
\newcommand{\eps}{\epsilon}
\newcommand{\ind}{\operatorname{ind}}

\newcommand{\Vol}{\operatorname{Vol}}
\newcommand{\PD}{\operatorname{PD}}

\newcommand{\ECH}{\operatorname{ECH}}
\newcommand{\ECC}{\operatorname{ECC}}
\newcommand{\HM}{\operatorname{HM}}
\newcommand{\NH}{\textup{(NH)}}

\title{Existence of a positive hyperbolic orbit in three-dimensional Reeb flows}
\author{Taisuke SHIBATA\thanks{Research Institute for Mathematical Sciences,
Kyoto University, Kyoto 606-8502, Japan.
E-mail address: \texttt{shibata@kurims.kyoto-u.ac.jp}.}}
\date{September 5, 2026}
\begin{document}
\maketitle

\begin{abstract}
Non-degenerate periodic orbits in three-dimensional Reeb flows are
classified into three types: positive hyperbolic, negative hyperbolic and
elliptic. In the present paper, we consider a closed connected contact
three-manifold with a non-degenerate contact form. We show that its Reeb
flow has a simple positive hyperbolic orbit if it has at least three
simple periodic orbits. We mainly study the case in which no elliptic
orbit exists. We prove that there is no  non-degenerate contact form on a closed
connected three-manifold with $b_1=0$ such that all simple periodic orbits are negative hyperbolic. As a corollary, by combining  the author's previous result in the presence of an
elliptic orbit and the known result for $b_1>0$, we obtain the main result.
The proof uses the Weyl law for ECH spectral invariants. We also use
compactness for genus zero $J$-holomorphic curves counted by the $U$-map.
Under the contrary assumption, the number of simple orbits in each action
interval $[L,2L]$ is uniformly bounded with respect to $L$. We use this property to study ECH
generators and genus zero $U$-curves.
\end{abstract}

\clearpage
\tableofcontents

\section{Introduction}
\subsection{Background and the statement of the main result}

Let $(Y,\lambda)$ be a closed connected contact three-manifold. We write
$\xi=\ker\lambda$. The Reeb vector field $X_\lambda$ is determined by
\[
 \lambda(X_\lambda)=1,\qquad d\lambda(X_\lambda,\cdot)=0.
\]
A periodic Reeb orbit is a map $\gamma:\R/T\Z\to Y$ such that
$\dot\gamma=X_\lambda(\gamma)$, where $T>0$. It is simple if this map is an
embedding. We consider two periodic orbits to be equivalent if they
define the same current.

Let $\phi^t$ denote the Reeb flow. A periodic orbit $\gamma$ of period $T$
is non-degenerate if
\[
 d\phi^T|_{\xi_{\gamma(0)}}:\xi_{\gamma(0)}\longrightarrow\xi_{\gamma(0)}
\]
has no eigenvalue equal to $1$. We call $\lambda$ non-degenerate if every
periodic orbit is non-degenerate.  A non-degenerate orbit is positive hyperbolic if the
eigenvalues of its linearized return map are positive real numbers. It is
negative hyperbolic if they are negative real numbers. It is elliptic if
they lie on the unit circle. 

The three-dimensional Weinstein conjecture states that every contact
form on a closed three-manifold has a periodic Reeb orbit.
Taubes proved this conjecture in \cite{TW}.
Hutchings and Taubes proved that there are at least two simple Reeb
orbits under the non-degeneracy assumption
\cite{HT3}.
Cristofaro-Gardiner and Hutchings removed this assumption
\cite{CH}.

We next recall the results on the number of simple periodic orbits.
Suppose that $\lambda$ is non-degenerate and $c_1(\xi)$ is torsion.
Cristofaro-Gardiner, Hutchings and Pomerleano proved that there are
either exactly two or infinitely many simple periodic orbits
\cite{CHP}.
Colin, Dehornoy and Rechtman removed the torsion assumption
\cite{CDR}.
Thus every non-degenerate contact form on a closed connected
three-manifold has either exactly two or infinitely many simple
periodic orbits.
In particular, if there are at least three simple periodic orbits,
then there are infinitely many simple periodic orbits.

The case of exactly two simple periodic orbits is also understood.
Hutchings and Taubes proved that, for a non-degenerate contact form,
both orbits are elliptic and $Y$ is $S^3$ or a lens space
\cite{HT3}.
They also proved that a non-degenerate contact form has exactly two
simple periodic orbits if all its simple periodic orbits are elliptic
\cite{HT3}.

Therefore, if a non-degenerate contact form has infinitely many simple
periodic orbits, at least one of them is hyperbolic. In addition,  we can realize a Reeb flow such that all simple orbits are positive hyperbolic as a contact Anosov flow.
As a refinement of the Weinstein conjecture, 
Cristofaro-Gardiner, Hutchings and Pomerleano asked the following
question \cite[Question 1.8]{CHP}.
Suppose that $Y$ is neither $S^3$ nor a lens space.
Does every non-degenerate contact form on $Y$ have a simple positive
hyperbolic orbit?
They proved that such an orbit exists when $b_1(Y)>0$
\cite[Proposition 1.9]{CHP}.

We call a non-degenerate Reeb flow strongly non-degenerate if
all intersections between the stable and unstable manifolds
of hyperbolic periodic orbits are transverse.
Colin, Dehornoy and Rechtman studied this question under this
assumption \cite[Section 4.5]{CDR}.
They proved the existence of a simple positive hyperbolic orbit
in several cases.

In joint work with Asaoka, the author proved the following result
\cite[Theorem 1.1]{AS}.
Suppose that a Reeb flow on a closed connected contact three-manifold
is strongly non-degenerate and has at least three simple periodic orbits.
Then it has infinitely many simple positive hyperbolic periodic orbits.

In this paper, we do not assume that the stable and unstable
manifolds intersect transversely.
We prove the existence of at least one simple positive hyperbolic orbit, which gives the affirmative answer to \cite[Question 1.8]{CHP}.

The main result is the following theorem.

\begin{theorem}\label{thm:main}\label{cor:main}
Let $(Y,\lambda)$ be a closed connected non-degenerate contact
three-manifold.
If $(Y,\lambda)$ has at least three simple periodic Reeb orbits,
then it has a simple positive hyperbolic periodic Reeb orbit.
\end{theorem}

Together with the results above, Theorem~\ref{thm:main} gives the
following result.
A non-degenerate contact form on a closed connected three-manifold
has either exactly two or infinitely many simple periodic orbits.
In the first case, both orbits are irrationally elliptic, and $Y$ is
$S^3$ or a lens space.
In the second case, at least one simple periodic orbit is positive
hyperbolic.

In \cite{S}, the author studied the case $b_1(Y)=0$.
Suppose that the flow has infinitely many simple periodic orbits and
at least one elliptic orbit.
Then it has a simple positive hyperbolic orbit.
Thus, to prove Theorem~\ref{thm:main}, it remains to exclude the case
in which $b_1(Y)=0$ and all simple periodic orbits are negative
hyperbolic.
The main step is to prove  the following theorem.

\begin{theorem}\label{thm:negative}
Let $Y$ be a closed connected three-manifold with $b_1(Y)=0$.
Then there is no non-degenerate contact form on $Y$
whose simple periodic orbits are all negative hyperbolic.
\end{theorem}

 The main part of this paper is to give the proof of Theorem \ref{thm:negative}.

\subsection{Idea of the proof and the structure of this paper}

We prove Theorem~\ref{thm:negative} by contradiction. Suppose that all simple
periodic orbits are negative hyperbolic. Every ECH generator
then has even ECH parity. Thus the differential vanishes. Moreover, the
multiplicity of each orbit in an ECH generator is one. Thus the generators
can be regarded as finite subsets of the set of simple periodic orbits.

We first use the Weyl law and the structure of ECH on a rational homology
sphere. They imply that the number of ECH generators of action less than
$L$ grows quadratically in $L$. We then add one orbit to an ECH generator
and count the obtained generators. This gives a uniform bound on the
number of simple orbits with action in $[L,2L]$. In particular, the number
of simple orbits of action at most $L$ is $O(\log L)$.

The main geometric argument is Lemma~\ref{lem:common}. Consider a genus
zero $U$-curve with small $d\lambda$-area and $J_0\le2$. Its nontrivial
component has at most four ends. However, this does not bound the number
of trivial cylinders in the current.We choose one of these common orbits and move the base point
of the $U$-map towards it. We use  low ECH index
compactness theorem. A nontrivially broken limit has
two embedded nontrivial parts of index one. The genus zero condition then gives a unique
connection between the two nontrivial components. The ECH partition condition and the parity of the ECH index
show that this connection is a double cover of the chosen
negative hyperbolic orbit. This gives a relation between its action and
the actions of the ends of the $J$-holomorphic curve counted by the $U$-map.  By using  this relation, we obtain a bound on the number of trivial cylinders.

Fix a sufficiently small $\epsilon>0$.
We first order all simple periodic orbits by their actions.
We divide them into disjoint subsets in this order.
We compare each orbit with the next orbit.
If the difference between their actions is less than $\epsilon$,
we put them in the same subset.
Otherwise, we start a new subset.
The number of orbits in each subset is uniformly bounded.

In each subset, we give the orbits the numbers $1,2,\ldots$
in the order of their actions.
For each ECH generator, we take the sum of these numbers.

Next, we take ECH generators $\alpha_0,\ldots,\alpha_K$ such that
$U\alpha_k=\alpha_{k-1}$ for $1\le k\le K$.
Consider a step from $\alpha_k$ to $\alpha_{k-1}$.
Suppose that only one orbit is replaced by another orbit with smaller action.
Suppose also that the difference between their actions is less than $\epsilon$.
Then these two orbits are in the same subset.
The new orbit has a smaller number.
The other orbits do not change.
Thus the sum decreases by at least one.

At the other steps, the sum may increase.
We use the bound on the number of trivial cylinders
and the estimate for $I-J_0$.
These give an upper bound on the sum of such increases.

Put $L=A(\alpha_K)$.
We then obtain $K\le CL$, where $C$ does not depend on $K$.
On the other hand, the Weyl law gives
\[
 \frac{L^2}{2K}\longrightarrow \operatorname{Vol}(Y,\lambda)>0.
\]
This is a contradiction.

In Section~\ref{sec:preliminaries}, we recall the facts about ECH used
in this paper. We also explain its relation with monopole Floer cohomology
and the tower structure on a rational homology sphere.
In Section~\ref{sec:actions}, we prove the action estimates
and define the local weights. In Section~\ref{sec:genuszero}, we study the
ends of genus zero $U$-curves. In Section~\ref{sec:common}, we prove the
bound on common orbits. We explain the compactness argument in detail.
In Section~\ref{sec:linear}, we prove the linear bound for the length of
a $U$-sequence. Finally, in Section~\ref{sec:mainproof}, we prove
Theorem~\ref{thm:negative} and Theorem~\ref{thm:main}.

\subsection*{Acknowledgments}

I am grateful for the help of AI tools in the preparation of this paper.
They helped me express my ideas more clearly in English.
They also helped me refine the arguments and present the proofs
more clearly.
I take full responsibility for the mathematical content of this paper.

Over the past few years, personal circumstances have left me
with little time for mathematics.
I apologize to my colleagues and others for the delays
and inconvenience this has caused.
I will address the remaining matters one by one.
I thank them for their patience and understanding.

\section{Preliminaries}\label{sec:preliminaries}
\subsection{Orbit sets and the ECH index}

Let $(Y,\lambda)$ be a closed connected non-degenerate contact
three-manifold. We use coefficients in $\F=\Z/2\Z$. An orbit set is a finite collection
$\alpha=\{(\alpha_i,m_i)\}$ of distinct simple Reeb orbits and positive
integers. We call $\alpha$ an ECH generator if $m_i=1$ whenever $\alpha_i$
is hyperbolic. Set
\[
 [\alpha]=\sum_i m_i[\alpha_i],\qquad
 \A(\alpha)=\sum_i m_i\int_{\alpha_i}\lambda.
\]
For $\Gamma\in H_1(Y;\Z)$, the chain group is
\begin{equation}\label{eq:eccdefinition}
 \ECC(Y,\lambda,\Gamma)
 =\bigoplus_{\substack{\alpha\text{ an ECH generator}\\{[\alpha]=\Gamma}}}
 \F\,\alpha.
\end{equation}
Note that the empty set  is an ECH generator.

For $[\alpha]=[\beta]$ with $\alpha=\{(\alpha_i,m_i)\}$ and $\beta=\{(\beta_j,n_j)\}$, let $H_2(Y,\alpha,\beta)$ be the affine space of
relative classes of integral two-chains with boundary $\sum_i m_i\alpha_i-\sum_j n_j\beta_j$.
It is an affine space over $H_2(Y;\Z)$. Choose a symplectic trivialization
$\tau$ of $\xi$ along these orbits. For $Z\in H_2(Y,\alpha,\beta)$, the
ECH index is
\begin{equation}\label{eq:echindex}
\begin{split}
 I(\alpha,\beta,Z)
 ={}&c_\tau(Z)+Q_\tau(Z)\\
 &+\sum_i\sum_{q=1}^{m_i}\mu_\tau(\alpha_i^q)
 -\sum_j\sum_{q=1}^{n_j}\mu_\tau(\beta_j^q).
\end{split}
\end{equation}
 Here $c_\tau$ is the relative first Chern
number, $Q_\tau$ is the relative intersection number, and $\mu_\tau$ is
the Conley--Zehnder index. The ECH index $I$ is independent of $\tau$ and is additive under concatenation.
The definition and basic properties are due to Hutchings
\cite{H1}.

A $\lambda$-compatible almost complex structure $J$ on $\R\times Y$ is
$\R$-invariant, satisfies $J\partial_s=X_\lambda$ and $J\xi=\xi$, and is
compatible with $d\lambda$ on $\xi$. We consider finite-energy
$J$-holomorphic curves with ends at covers of Reeb orbits. A $J$-holomorphic
current is a finite collection of distinct irreducible somewhere
injective curves with positive integer multiplicities. We denote the
space of currents from $\alpha$ to $\beta$ by $\Mcal^J(\alpha,\beta)$. Let $\pi_{Y}: \mathbb{R}\times Y \to Y$ be the projection. Then,  we can define the ECH index of  $C \in \Mcal^J(\alpha,\beta)$ by $I(\alpha,\beta,[\pi_{Y}(C)])$ where $[\pi_{Y}(C)] \in H_2(Y,\alpha,\beta)$ is naturally defined.

For a somewhere injective curve $u$ with punctured domain $\dot\Sigma$, Fredholm index is 
\begin{equation}\label{eq:fredholm}
 \ind(u)=-\chi(\dot\Sigma)+2c_\tau(u)
 +\sum_{z\in\Gamma^+}\mu_\tau(\gamma_z)
 -\sum_{z\in\Gamma^-}\mu_\tau(\gamma_z).
\end{equation}

If $J$ is generic and $u$ is somewhere injective,
the moduli space is smooth near $u$ and has dimension $\ind(u)$.
See \cite[Proposition 1.4]{HT1}.

\subsection{The differential and the \texorpdfstring{$U$}{U}-map}

Let $\Mcal^J_r(\alpha,\beta)$ denote the currents of ECH index $r$.
For generic $J$, define
\begin{equation}\label{eq:differentialdefinition}
 \partial_J\alpha
 =\sum_{\substack{\beta\text{ an ECH generator}\\{[\beta]=\Gamma}}}
 \#\bigl(\Mcal^J_1(\alpha,\beta)/\R\bigr)\,\beta.
\end{equation}
Here $\#$ denotes the count modulo two, and $\R$ acts by translation.
The count is finite and $\partial_J^2=0$;
see \cite[Lemma 7.19 and Theorem 7.20]{HT1}.
The gluing argument is completed in \cite{HT2}.
We set
\begin{equation}\label{eq:echhomology}
 \ECH_*(Y,\lambda,\Gamma;J)
 =H_*(\ECC(Y,\lambda,\Gamma),\partial_J).
\end{equation}

The relative grading takes values in $\Z/d_\Gamma\Z$. Here $d_\Gamma$ is
the divisibility of $c_1(\xi)+2\PD(\Gamma)$ modulo torsion.
If this class is torsion, we set $d_\Gamma=0$.
In this case the relative grading is integer-valued and, we choose an integer grading
$\operatorname{gr}$ such that
\[
 \operatorname{gr}(\alpha)-\operatorname{gr}(\beta)
 =I(\alpha,\beta,Z)
\]
for ECH generators $\alpha,\beta$ in class $\Gamma$.
This grading is determined up to an additive constant.
For a nonzero homogeneous ECH class $\sigma$,
we also write $\operatorname{gr}(\sigma)$ for its degree.
The differential has degree $-1$. The parity of a generator
is the number of its positive hyperbolic orbits modulo two (see
\cite{H1}). Hence, if
there are no simple positive hyperbolic orbits, then
\begin{equation}\label{eq:zero}
 \partial_J=0.
\end{equation}

For a generic point $z$ not on any periodic orbits, define
\begin{equation}\label{eq:udefinition}
 \langle U_{J,z}\alpha,\beta\rangle
 =\#\{\Ccal\in\Mcal^J_2(\alpha,\beta):(0,z)\in\Ccal\}.
\end{equation}
This is a chain map of degree $-2$.
For two generic points $z,z'$ outside the periodic orbits,
there is a chain homotopy $H$ such that
\[
 U_{J,z}-U_{J,z'}=\partial_JH+H\partial_J.
\]
See \cite{HT3}.
Thus the induced map on ECH is independent of $z$.
We denote it by $U_J$.
If $\partial_J=0$, then the chain map $U_{J,z}$ itself is
independent of $z$.
We write $U$ for the chain map when $J$ and $z$ are fixed.

\paragraph{Curves of low ECH indices.}
Write a holomorphic curve as $u=u_0\cup u_1$. The part $u_0$ consists
of covers of trivial cylinders. The part $u_1$ has no such component.
For generic $J$, the ECH index is nonnegative.
We recall the low-index classification
\cite{H1};
see also \cite{HT1}.
\begin{enumerate}[label=(\roman*)]
\item If $I(u)=0$, then $u_1=\varnothing$.
\item If $I(u)=1$, then $u_1$ is embedded, connected and
$\ind(u_1)=I(u_1)=1$. In addition, $u_0$ and $u_1$ are disjoint.
\item If $I(u)=2$ and the positive and negative orbit sets are ECH
 generators, then $u_1$ is embedded and $\ind(u_1)=I(u_1)=2$. In addition, $u_0$ and $u_1$ are disjoint.
\end{enumerate}
In (iii), $u_1$ need not be connected.
If it is disconnected, it has exactly two components.
Each component has Fredholm index one.

For a $J$-holomorphic curve counted by $U_{J,z}$, the nontrivial part is connected
by \cite[Lemma 2.6(b)]{HT3}.
Thus it has the form
\begin{equation}\label{eq:split}
 \Ccal=\Ccal_0\sqcup C,
\end{equation}
where $\Ccal_0$ consists of trivial cylinders and $C$ is an embedded
irreducible nontrivial curve with
\begin{equation}\label{eq:indtwo}
 I(C)=\ind(C)=2.
\end{equation}
The images of $C$ and $\Ccal_0$ are disjoint, and $(0,z)\in C$.

\paragraph{Partition conditions.}
For a simple negative hyperbolic orbit $\gamma$, the incoming and
outgoing partitions of a positive integer $m$ are given by
\[
 P^{\mathrm{in}}_\gamma(m)=P^{\mathrm{out}}_\gamma(m)=
 \begin{cases}
  (2,\ldots,2),&m\text{ even},\\
  (2,\ldots,2,1),&m\text{ odd}.
 \end{cases}
\]
(See \cite{H1} and
\cite{HT1}).
If a nontrivial somewhere injective curve $u$ satisfies
$I(u)=\ind(u)$, its end multiplicities are given by these partitions (see \cite{H1,HT1}).
For such a curve, each end at a cover of $\gamma$ has
multiplicity one or two. 

\subsection{The isomorphism with monopole Floer cohomology}
\label{subsec:monopole}

We use the monopole Floer cohomology $\widehat{\HM}^{*}$ of
Kronheimer and Mrowka \cite{KM}. It is defined from the perturbed
Seiberg--Witten equations. Let $\mathfrak s_\xi$ be the
$\operatorname{spin}^c$ structure of $\xi$. For $\Gamma\in H_1(Y;\Z)$,
set $\mathfrak s_\Gamma=\mathfrak s_\xi+\PD(\Gamma)$. Then
\begin{equation}\label{eq:spincchern}
 c_1(\mathfrak s_\Gamma)=c_1(\xi)+2\PD(\Gamma).
\end{equation}

\begin{theorem}[Taubes {\cite[Theorem 1]{T}} and {\cite[Theorem 1.1]{TV}}]
\label{thm:taubes}
For every $\Gamma$ and generic $\lambda$-compatible $J$, there is a
canonical isomorphism
\begin{equation}\label{eq:taubesisomorphism}
 \Psi_\Gamma:\ECH_*(Y,\lambda,\Gamma;J)
 \xrightarrow{\simeq}
 \widehat{\HM}^{-*}(Y,\mathfrak s_\Gamma).
\end{equation}

satisfying
\begin{equation}\label{eq:uequivariance}
 \Psi_\Gamma\circ U_J
 =U_{\mathrm{HM}}\circ\Psi_\Gamma.
\end{equation}
\end{theorem}
Here $U_{\mathrm{HM}}$ has  degree $2$. ECH and $U_J$ do not depend on $J$, hence we just write them as $\ECH_*(Y,\lambda,\Gamma)$ and $U$.

\subsection{The $U$ tower on a rational homology sphere}
\label{subsec:rationaltower}

Suppose that $Y$ is a rational homology sphere, equivalently $b_1(Y)=0$.
Then $H_1(Y;\Z)$ is finite and $H_2(Y;\Z)=0$. All $c_1(\mathfrak s_\Gamma)$
are torsion, so we choose an integer grading in each class $\Gamma$.
Define the graded module
\begin{equation}\label{eq:tplus}
 \mathcal T^+
 =\F[U,U^{-1}]/U\F[U],
 \qquad \deg U=-2.
\end{equation}
Let $\mathcal T^+_d$ have its bottom element in degree $d$. Its basis
$e_k=U^{-k}$ satisfies
\begin{equation}\label{eq:abstracttower}
 \begin{gathered}
 \deg e_k=d+2k,\qquad k\ge0,\\
 U e_0=0,\qquad U e_{k+1}=e_k.
 \end{gathered}
\end{equation}

\begin{proposition}\label{prop:rhstower}
For each $\Gamma\in H_1(Y;\Z)$, there are an integer $d$ and a
finite-dimensional graded $\F[U]$-module $R_\Gamma$ such that
\begin{equation}\label{eq:echdecomposition}
 \ECH_*(Y,\lambda,\Gamma)
 \cong\mathcal T^+_d\oplus R_\Gamma
\end{equation}
as graded $\F[U]$-modules.
\end{proposition}

\begin{proof}
This follows from Theorem~\ref{thm:taubes} and the structure of
monopole Floer cohomology for rational homology spheres;
see \cite[Propositions 35.3.1 and 22.2.3]{KM}.
\end{proof}

\begin{corollary}\label{cor:generatortower}
Suppose that $Y$ is a rational homology sphere and that $\lambda$ has no
simple positive hyperbolic orbit. For fixed generic $J$ and each $\Gamma$, there are an integer
$a_\Gamma$ and ECH generators $\{\alpha_k^\Gamma\}_{k\ge0}$ such that
\begin{equation}\label{eq:generatordegrees}
 \operatorname{gr}(\alpha_k^\Gamma)=a_\Gamma+2k
\end{equation}
and  for every generic point $z$,
\begin{equation}\label{eq:tower}
 U_{J,z}\alpha_{k+1}^\Gamma=\alpha_k^\Gamma,
 \qquad I(\alpha_{k+1}^\Gamma,\alpha_k^\Gamma)=2.
\end{equation}
Only finitely many ECH generators in class $\Gamma$ are not in  this
sequence.
\end{corollary}

\begin{proof}
By \eqref{eq:zero}, $\partial_J=0$, so the chain group equals its homology.
Choose $a_\Gamma$ in the tower parity, above the bottom of the tower and
the degrees of $R_\Gamma$. Each group in degree $a_\Gamma+2k$ is
one-dimensional. Hence it contains exactly one ECH generator
$\alpha_k^\Gamma$. In these degrees, $U$ is an isomorphism.
Its chain-level coefficient is therefore $1$ over $\F$.
The remaining degrees have finite total dimension, so only finitely many ECH generators in class $\Gamma$ are not in  this
sequence.
\end{proof}

\subsection{The \texorpdfstring{$J_0$}{J0} index}

We use the index $J_0$ introduced by Hutchings
\cite[Definition 6.2]{H2}.
Let $\alpha=\{(\alpha_i,m_i)\}$ and $\beta=\{(\beta_j,n_j)\}$
be orbit sets with $[\alpha]=[\beta]$.
For $Z\in H_2(Y,\alpha,\beta)$, define
\begin{equation}\label{eq:jzero}
\begin{split}
 J_0(\alpha,\beta,Z)
 ={}&-c_\tau(Z)+Q_\tau(Z)\\
 &+\sum_i\sum_{q=1}^{m_i-1}\mu_\tau(\alpha_i^q)
 -\sum_j\sum_{q=1}^{n_j-1}\mu_\tau(\beta_j^q).
\end{split}
\end{equation}
The index does not depend on $\tau$.
It is additive under concatenation
\cite[Proposition 6.5]{H2}.
For a current $\Ccal\in\Mcal^J(\alpha,\beta)$, we write
$J_0(\Ccal)=J_0(\alpha,\beta,[\pi_Y(\Ccal)])$.

We use $J_0$ to estimate the genus and the number of ends of
a curve counted by the $U$-map.
Suppose that all multiplicities in $\alpha$ and $\beta$ are one.
Let $\Ccal=\Ccal_0\sqcup C$ be a current from $\alpha$ to $\beta$
counted by $U_{J,z}$.
Let $g(C)$ and $e(C)$ be the genus and the total number of ends of $C$.
By \cite{CHP}, we have
\begin{equation}\label{eq:topology}
 J_0(\Ccal)=-\chi(C)=2g(C)-2+e(C).
\end{equation}
Since $C$ has at least one positive end, we obtain
\begin{equation}\label{eq:jlower}
 J_0(\Ccal)\ge-1.
\end{equation}

\begin{proposition}[{\cite{CHP,CP}}]
\label{prop:ij}
Suppose that $c_1(\xi)$ is torsion. For fixed $\lambda$, there is a
constant $C_*>0$ such that
\begin{equation}\label{eq:ij}
 |I(\alpha,\beta,Z)-J_0(\alpha,\beta,Z)|
 \le C_*\A(\alpha)
\end{equation}
for all ECH generators with $[\alpha]=[\beta]$ and
$\A(\beta)\le\A(\alpha)$, and all $Z\in H_2(Y,\alpha,\beta)$.
\end{proposition}

\subsection{ECH spectral invariants and the Weyl law}

Let $\ECC^L$ be generated by ECH generators of action less than $L$.
The differential and $U$ decrease action, since
\[
 \A(\alpha)-\A(\beta)=\int_{\Ccal}d\lambda>0
\]
for any current counted by these maps. Thus $\ECC^L$ is a subcomplex.
Write $\ECH^L$ for its homology and $\iota_L:\ECH^L\to\ECH$ for the
induced map by the inclusion. For a nonzero homogeneous class $\sigma$, define
\begin{equation}\label{eq:spectraldefinition}
 c_\sigma(Y,\lambda)
 =\inf\{L>0:\sigma\in\operatorname{Im}(\iota_L)\}.
\end{equation}
This is \cite[Definition 4.1]{H3}.
Set $\Vol(Y,\lambda)=\int_Y\lambda\wedge d\lambda$.
The Weyl law \cite[Theorem 1.3]{CHR} states that
\begin{equation}\label{eq:volume}
 \lim_{k\to\infty}
 \frac{c_{\sigma_k}(Y,\lambda)^2}{\operatorname{gr}(\sigma_k)}
 =\Vol(Y,\lambda)
\end{equation}
if $c_1(\xi)+2\PD(\Gamma)$ is torsion and the nonzero homogeneous classes
$\sigma_k\in\ECH(Y,\lambda,\Gamma)$ satisfy
$\operatorname{gr}(\sigma_k)\to\infty$. 

\section{Action estimates in the negative hyperbolic case}\label{sec:actions}

From now on, until Section~\ref{sec:mainproof}, we suppose that
\begin{equation}\label{eq:nh}
 \NH:\quad b_1(Y)=0\quad\text{and all simple Reeb orbits are negative hyperbolic}.
\end{equation}
The form $\lambda$ is always assumed to be non-degenerate.

The manifold $Y$ is an oriented rational homology sphere. Thus
$H_1(Y;\Z)$ is finite and $H_2(Y;\Z)=0$. In particular, $c_1(\xi)$ is
torsion. For homologous orbit sets, the relative class is unique. Hence
we just write $I(\alpha,\beta)$ and $J_0(\alpha,\beta)$ without the relative
class. Every ECH generator is now a finite set of simple orbits. We
denote its cardinality by $|\alpha|$.

 We may multiply $\lambda$ by a positive constant and assume that
\begin{equation}\label{eq:minperiod}
 T_\gamma:=\A(\gamma)\ge1
 \quad\text{for every simple periodic orbit }\gamma.
\end{equation} 

\subsection{Counting ECH generators}

For $L>0$ and $\Gamma\in H_1(Y;\Z)$, use the notation
\[
 \Lambda(L,\Gamma)=
 \{\alpha:\alpha\text{ is an ECH generator},\ [\alpha]=\Gamma,\ \A(\alpha)<L\}.
\]
Let
\[
 Q(L)=\sum_{\Gamma\in H_1(Y;\Z)}|\Lambda(L,\Gamma)|,
 \qquad
 N(L)=\#\{\gamma\text{ simple}:T_\gamma\le L\}.
\]

\begin{lemma}\label{lem:count}
Under \NH, for every $\Gamma\in H_1(Y;\Z)$,
\begin{equation}\label{eq:countgamma}
 \lim_{L \to \infty}\frac{L^2}{|\Lambda(L,\Gamma)|}=2\Vol(Y,\lambda).
\end{equation}
Consequently,
\begin{equation}\label{eq:countall}
\lim_{L \to \infty}\frac{L^2}{Q(L)}=\frac{2\Vol(Y,\lambda)}{|H_1(Y;\Z)|}.
\end{equation}
\end{lemma}

\begin{proof}[\bf 
Proof of Lemma~\ref{lem:count}]
Fix $\Gamma$. Apply Corollary~\ref{cor:generatortower} and write the
resulting ECH generators as $\alpha_k^\Gamma$. Set
$a_k^\Gamma=\A(\alpha_k^\Gamma)$. Equation \eqref{eq:tower} gives a
$U$-current from $\alpha_{k+1}^\Gamma$ to $\alpha_k^\Gamma$. Its
nontrivial component has positive $d\lambda$-area. Thus
$a_{k+1}^\Gamma>a_k^\Gamma$.

Let $\sigma_k^\Gamma=[\alpha_k^\Gamma]$ denote the corresponding ECH
class. Since $\partial_J=0$, an ECH generator is the only cycle which
represents its ECH class. Hence \eqref{eq:spectraldefinition} gives
\[
 c_{\sigma_k^\Gamma}(Y,\lambda)=a_k^\Gamma.
\]
The degree is $a_\Gamma+2k$ by \eqref{eq:generatordegrees}. The Weyl law
therefore gives
\begin{equation}\label{eq:towerasymp}
 \lim_{k\to\infty}\frac{(a_k^\Gamma)^2}{2k}=\Vol(Y,\lambda).
\end{equation}
The fixed grading shift $a_\Gamma$ does not affect this limit.

There are only finitely many ECH generators outside this sequence.
Consequently,
\[
 |\Lambda(L,\Gamma)|=\#\{k\ge0:a_k^\Gamma<L\}+O(1).
\]
This gives \eqref{eq:countgamma}.  We sum over $\Gamma\in H_1(Y;\Z)$ and obtain
\eqref{eq:countall}.
\end{proof}

\subsection{A uniform bound in each interval of actions}

\begin{lemma}\label{lem:dyadic}
There is an integer $B\ge1$ such that
\begin{equation}\label{eq:dyadic}
 \#\{\gamma\text{ simple}:x\le T_\gamma\le2x\}\le B
 \qquad\text{for every }x\ge\tfrac12.
\end{equation}
In particular,
\begin{equation}\label{eq:logarithmic}
 N(L)=O(\log L).
\end{equation}
\end{lemma}

\begin{proof}[\bf Proof of Lemma~\ref{lem:dyadic}]
Denote the cardinality  \eqref{eq:dyadic} by $r(x)$.
If $\A(\alpha)<x$ and $x\le T_\gamma\le2x$, then
$\gamma\notin\alpha$ and $\alpha\cup\{\gamma\}$ is an ECH generator. In addition, the map
\[
 (\alpha,\gamma)\longmapsto\alpha\cup\{\gamma\}
\]
is injective and the action of the image is less than $3x$.
Hence
\[
 r(x)Q(x)\le Q(3x).
\]
By \eqref{eq:countall}, $Q(3x)/Q(x)\to9$. Thus $r(x)$ is uniformly
bounded for large $x$.
For bounded $x$, we use the finiteness of simple orbits with bounded action.
We increase $B$ if necessary. Then \eqref{eq:dyadic} holds for all $x\ge1/2$.

In particular, this implies that $N(2^n) \leq nB$ and hence $N(L)\leq C \log L$ for constant $C=C(B)$.
As a result, \eqref{eq:logarithmic} follows.
\end{proof}

\subsection{Ordering orbits with close actions}

From now on, fix
\begin{equation}\label{eq:epsilon}
 0<\eps<\frac{1}{4(B+1)}.
\end{equation}
In particular, $\eps<1/4$. 

Consider a graph whose vertices are the simple periodic orbits. We join
two distinct vertices by an edge if their actions differ by less than
$\eps$.

\begin{lemma}\label{lem:weights}
Every connected component of this graph has at most $B$ vertices.
Consequently, there is a function $h$ on the set of simple periodic orbits
such that
\begin{equation}\label{eq:weightrange}
 1\le h(\gamma)\le B
\end{equation}
and
\begin{equation}\label{eq:weightorder}
 0<T_\gamma-T_\delta<\eps
 \quad\Longrightarrow\quad h(\gamma)-h(\delta)\ge1.
\end{equation}
\end{lemma}

\begin{proof}[\bf Proof of Lemma~\ref{lem:weights}]
Suppose that one component has at least $B+1$ vertices. The actions are
at least $1$ and are locally finite. We list the first $B+1$ actions as
\[
 t_1\le t_2\le\cdots\le t_{B+1}.
\]
Note that for each $i$, one has $t_{i+1}-t_i<\eps$. Therefore
\[
 t_{B+1}-t_1<B\eps<\tfrac14<t_1.
\]

Thus $t_{B+1}<2t_1$.

Then the interval $[t_1,2t_1]$ contains $B+1$ simple orbits. This
contradicts Lemma~\ref{lem:dyadic}. Thus each component has at most $B$
vertices.

In each component, we assign numbers in order of action. (If two actions
are equal, we choose either order.) The numbers are in between $1$ and $B$.
Two orbits which satisfy the hypothesis of \eqref{eq:weightorder} are in
the same component.
This proves the lemma.
\end{proof}

\section{The ends of genus zero curves}\label{sec:genuszero}

We next study the nontrivial component of a $U$-current when its genus
is zero.
\begin{lemma}\label{lem:projection}
Let $\Ccal=\Ccal_0\sqcup C$ be a current counted by the $U$-map under
\NH. If $g(C)=0$, then $\pi_Y:C\to Y$ is an embedding.
\end{lemma}

\begin{proof}[\bf Proof of Lemma~\ref{lem:projection}]
Every end of $C$ is simply covered and negative hyperbolic.
Every curve in the same moduli component has the same ends
and relative class.
Its ECH index is  two, and  every such curve is embedded
in $\R\times Y$.

Since $g(C)=0$ and $\ind(C)=2$, the normal Chern number is zero.
The other hypotheses of \cite[Proposition 3.3]{CHP}
follow from multiplicity one.
Hence $\pi_Y:C\to Y$ is an embedding.
\end{proof}

\begin{lemma}\label{lem:oppositeends}
Under the assumptions of Lemma~\ref{lem:projection}, the same simple orbit
cannot occur as both a positive end and a negative end of $C$.
\end{lemma}

\begin{proof}[\bf Proof of Lemma~\ref{lem:oppositeends}]
Suppose that a simple negative hyperbolic orbit $\eta$ is both
a positive and a negative end of $C$.
Fix a trivialization $\tau$ along $\eta$.
Since $\eta$ is negative hyperbolic, we can write
$\mu_\tau(\eta)=2q+1$ for some $q\in\Z$.
Let $w_+$ and $w_-$ be the winding numbers of the positive and
negative ends with respect to $\tau$.
By \cite[Lemmas 6.4(c) and 6.6]{H1}, we have
\begin{equation}\label{eq:winding}
 w_+\le q,\qquad w_-\ge q+1.
\end{equation}

We take a sufficiently small torus around $\eta$.
The projections of the two ends meet this torus in closed curves.
We orient both curves in the direction of $\eta$.
Since both ends are simply covered, their homology classes are
$(1,w_+)$ and $(1,w_-)$ in the coordinates given by $\tau$.
Their algebraic intersection number is nonzero since $w_+\ne w_-$.
Hence the two curves intersect.
This contradicts Lemma~\ref{lem:projection}, since $\pi_Y(C)$ is embedded.
\end{proof}

Let $\alpha,\beta$ be ECH generators with $[\alpha]=[\beta]$. Put
\begin{equation}\label{eq:pms}
 P=\alpha\setminus\beta,\qquad M=\beta\setminus\alpha,
 \qquad S=\alpha\cap\beta,
 \qquad d=|P|+|M|.
\end{equation}
Each common orbit has multiplicity one. Hence it is either a trivial
cylinder or both a positive and a negative end of the nontrivial
component. Let $r$ be the number of common orbits of the second type.
Then $e(C)=d+2r$. Thus \eqref{eq:topology} becomes
\begin{equation}\label{eq:gr}
 J_0(\alpha,\beta)=2g(C)-2+d+2r.
\end{equation}

\begin{corollary}\label{cor:fixedends}
Suppose that one current counted by $\langle U\alpha,\beta\rangle$ has
a nontrivial component of genus zero. Then, for every generic base point,
each current counted for this pair has the form
\begin{equation}\label{eq:commonfixed}
 \Ccal=\left(\bigsqcup_{\eta\in S}\R\times\eta\right)\sqcup C,
 \qquad C:P\longrightarrow M,
\end{equation}
where $C$ has genus zero. Moreover,
\begin{equation}\label{eq:jfixed}
 J_0(\alpha,\beta)=d-2.
\end{equation}
\end{corollary}

\begin{proof}[\bf Proof of Corollary~\ref{cor:fixedends}]
For the given genus zero current, Lemma~\ref{lem:oppositeends} gives
$r=0$. Equation \eqref{eq:gr} then gives \eqref{eq:jfixed}. We apply \eqref{eq:gr} to any other
current for the same pair. Then
\[
 d-2=2g(C)-2+d+2r.
\]
Thus $g(C)=r=0$, and \eqref{eq:commonfixed} follows. 
\end{proof}

\section{A uniform bound on common orbits}\label{sec:common}
In this section, we prove a uniform
bound on the number of  the trivial cylinders in a $J$-holomorphic curve counted by the $U$-map

\begin{lemma}\label{lem:common}
Assume \NH. Suppose that
\[
 \langle U\alpha,\beta\rangle=1,
 \qquad 0<\A(\alpha)-\A(\beta)<\eps,
 \qquad J_0(\alpha,\beta)\le2.
\]
If a $J$-holomorphic  curve counted by $\langle U\alpha,\beta\rangle=1$ has a genus zero nontrivial
component, then
\begin{equation}\label{eq:commonbound}
 |\alpha\cap\beta|\le16B,
 \qquad |\alpha|,|\beta|\le M_0:=16B+4.
\end{equation}
Note that the constants $B$, $\eps$ and $M_0$ are independent of $\alpha$ and $\beta$.
\end{lemma}

\begin{proof}[\bf Proof of Lemma~\ref{lem:common}]
We use the notation $P,M,S,d$ in \eqref{eq:pms}.
By Corollary~\ref{cor:fixedends}, we have
$J_0(\alpha,\beta)=d-2$. Hence $d\le4$.
The nontrivial component has at least one positive end.
If it had no negative end, its $d\lambda$-area would be at least one.
This contradicts $\A(\alpha)-\A(\beta)<\eps<1$.
Therefore
\begin{equation}\label{eq:dfour}
 2\le d\le4.
\end{equation}

If $S=\varnothing$, the lemma follows.
Suppose that $S\ne\varnothing$, and fix $\eta\in S$.
We first show that there are subsets $P'\subset P$ and $M'\subset M$
such that
\begin{equation}\label{eq:targetinterval}
 0<\A(P')-\A(M')-2T_\eta<\eps.
\end{equation}
We then use Lemma~\ref{lem:dyadic} to count the orbits in $S$.

Take $x\in\eta$.
Let $z_n$ be generic points outside all periodic orbits such that
$z_n\to x$.
Since $\partial_J=0$, the chain map $U_{J,z}$ is independent of $z$.
Thus
\[
 \langle U_{J,z_n}\alpha,\beta\rangle=1
 \qquad\text{for every }n.
\]
For each $n$, take a current $\Ccal_n$ counted by this coefficient.
By Corollary~\ref{cor:fixedends}, we can write
\[
 \Ccal_n=
 \left(\bigsqcup_{\zeta\in S}\R\times\zeta\right)\sqcup C_n,
 \qquad C_n:P\longrightarrow M,
 \qquad g(C_n)=0.
\]
The point $(0,z_n)$ lies on $C_n$.
Also, $C_n$ is disjoint from every trivial cylinder in $\Ccal_n$.
In particular,
\begin{equation}\label{eq:avoid}
 C_n\cap(\R\times\eta)=\varnothing.
\end{equation}

We now consider the limit of $C_n$ as $z_n\to x$.
The ends $P,M$ and the genus are fixed.
Hence the Hofer energies are uniformly bounded.
Also, $I(C_n)=\ind(C_n)=2$.
By the compactness theorem for curves of low ECH index
\cite[Theorem 1.8(a)(ii)]{H1}
(see also \cite[Lemma 7.23]{HT1}),
we may take a subsequence and translations so that $C_n$
converges to a possibly broken curve.
Here $P,M$ are ECH generators, and $C_n$ does not contain any trivial cylinder
component.

Recall that $(0,z_n)\in C_n$ and $z_n\to x\in\eta$.
By \eqref{eq:avoid} and positivity of intersections, a nontrivial
component of the limit has an end at a cover of $\eta$.
Since $\eta\notin P\cup M$, the limit is broken.

By \cite[Proposition 7.15(c) and Lemma 7.23]{HT1}, the limit has two
embedded nontrivial components $C_\eta^+$ and $C_\eta^-$ of genus zero.
They lie in the top and bottom levels, respectively, and satisfy
\[
 I(C_\eta^+)=\ind(C_\eta^+)=1,
 \qquad I(C_\eta^-)=\ind(C_\eta^-)=1.
\]
All other components are covers of trivial cylinders.
The total Fredholm index is two.
By additivity and \cite[Lemma 1.7]{HT1}, these covers have
Fredholm index zero.
Since all Reeb orbits are hyperbolic, they are unbranched.

The arithmetic genus of the whole limit is zero, since $g(C_n)=0$.
Hence its dual graph is a tree.
After we omit the trivial cylinders, there is exactly one connection
between $C_\eta^+$ and $C_\eta^-$.
Since $\eta\notin P\cup M$, this connection is a cover of $\eta$.
We write it as $\eta^m$, where $m\ge1$.
Here $\eta^m$ is a negative end of $C_\eta^+$
and a positive end of $C_\eta^-$.
All other ends of the two components are external and simply covered.

There is only one end of $C_\eta^+$ at a cover of $\eta$.
Since $I(C_\eta^+)=\ind(C_\eta^+)=1$, the partition condition gives
$m\in\{1,2\}$.
Suppose that $m=1$.
Then the orbit sets at the positive and negative ends of $C_\eta^+$
are both ECH generators.
They contain only negative hyperbolic orbits.
Hence $I(C_\eta^+)$ is even by
\cite[Proposition 1.6(c)]{H1}.
This contradicts $I(C_\eta^+)=1$.
Therefore $m=2$.

The component $C_\eta^+$ has $\eta^2$ as a negative end.
Its other positive ends form a subset $P'\subset P$.
Its other negative ends form a subset $M'\subset M$.
By Stokes' theorem,
\begin{equation}\label{eq:arearelation}
 \int_{C_\eta^+}d\lambda
 =\A(P')-\A(M')-2T_\eta.
\end{equation}
Both $C_\eta^+$ and $C_\eta^-$ have positive $d\lambda$-area.
Trivial cylinders have zero $d\lambda$-area. Hence
\begin{equation}
 0<\int_{C_\eta^+}d\lambda
 <\A(\alpha)-\A(\beta)<\eps.
\end{equation}
This proves \eqref{eq:targetinterval}.

We now count the orbits in $S$.
For $P'\subset P$ and $M'\subset M$, put
\[
 s(P',M')=\A(P')-\A(M').
\]
Since $\eta\in S$ was arbitrary, every orbit in $S$ satisfies
\begin{equation}\label{eq:shortinterval}
 T_\eta\in
 \left(\frac{s(P',M')-\eps}{2},\frac{s(P',M')}{2}\right)
\end{equation}
for some pair $(P',M')$.
The sets $P,M$ are fixed, and $|P|+|M|\le4$.
Therefore the number of these intervals is at most
\[
 2^{|P|}2^{|M|}=2^d\le16.
\]

Consider one such interval which contains an action $T_\eta\ge1$.
Put $a=(s(P',M')-\eps)/2$.
Then $a+\eps/2>T_\eta\ge1$, so
\[
 a>1-\eps/2>1/2.
\]
Since $\eps<1/4$, we also have
\[
 a+\eps/2<2a.
\]
Hence this interval is contained in $[a,2a]$.
By Lemma~\ref{lem:dyadic}, at most $B$ simple orbits have actions in it.
There are at most 16 intervals.
Therefore $|S|\le16B$.
Finally,
\[
 |\alpha|=|P|+|S|\le4+16B,
 \qquad |\beta|=|M|+|S|\le4+16B.
\]
This completes the proof.
\end{proof}

\section{A linear action bound for the length of a \texorpdfstring{$U$}{U}-sequence}
\label{sec:linear}

We now combine Lemma~\ref{lem:common} with the bounded weights from
Lemma~\ref{lem:weights}.

\begin{proposition}\label{prop:linear}
Assume \NH. Fix a sequence of ECH generators
$\{\alpha_k\}_{k\ge0}$ in one homology class and in sufficiently high
degrees. Suppose that $U\alpha_k=\alpha_{k-1}$ for $k\ge1$.
Put $L=\A(\alpha_K)$. Then there is a constant $C>0$, independent of $K$,
such that
\begin{equation}\label{eq:linearbound}
 K\le CL
\end{equation}
for all sufficiently large $K$.
\end{proposition}

\begin{proof}[\bf Proof of Proposition~\ref{prop:linear}]
Fix one generic point $z$. For each $1\le k\le K$, choose a current
$\Ccal_k=\Ccal_{0,k}\sqcup C_k$ counted by
$\langle U_{J,z}\alpha_k,\alpha_{k-1}\rangle=1$. Set
\[
 E_k=\A(\alpha_k)-\A(\alpha_{k-1})>0,
 \qquad j_k=J_0(\Ccal_k).
\]
Write $g_k=g(C_k)$ and $e_k=e(C_k)$. We proceed in five steps.

\medskip
\noindent\textbf{Step 1. The action and index estimates.}
Note that
\begin{equation}\label{eq:energybudget}
 \sum_{k=1}^K E_k
 =\A(\alpha_K)-\A(\alpha_0)\le L.
\end{equation}

Additivity gives $I(\alpha_{K},\alpha_{0})=2K$ and $J_0(\alpha_{K},\alpha_{0})=\sum_kj_k$. Therefore it follows from Proposition~\ref{prop:ij} that 
\begin{equation}\label{eq:jbudget}
 \left|\sum_{k=1}^K(j_k-2)\right|\le C_*L.
\end{equation}

Also, \eqref{eq:topology} gives
\begin{equation}\label{eq:stepj}
 j_k=2g_k-2+e_k\ge-1.
\end{equation}

\medskip
\noindent\textbf{Step 2. The low-action steps.}
Suppose that $E_k<\eps$ and $j_k\le2$. Since $E_k<1$, the curve $C_k$
has a negative end as well as a positive end. Hence $e_k\ge2$.
The inequality $2g_k-2+e_k\le2$ gives the following two cases:
\[
 g_k=0,\quad e_k\le4;
 \qquad\text{or}\qquad
 g_k=1,\quad e_k=2,\quad j_k=2.
\]
In the second case, the curve has one positive and one negative end.
These are distinct simple orbits. Thus
\begin{equation}\label{eq:exchange}
\begin{split}
 \alpha_k&=S_k\cup\{\gamma_+\},\qquad
 \alpha_{k-1}=S_k\cup\{\gamma_-\},\\
 &0<T_{\gamma_+}-T_{\gamma_-}<\eps.
\end{split}
\end{equation}
Let $\Gcal$ be the set of indices for these genus one steps with $E_k<\eps$,
and let
\[
 \Bcal=\{1,\ldots,K\}\setminus\Gcal.
\]
We call the indices in $\Bcal$ exceptional. Every step in $\Gcal$ has
$j_k=2$. 

\medskip
\noindent\textbf{Step 3. Counting the exceptional steps with $j_k\le2$.}

Put
\[
 D=\#\{k\in\Bcal:j_k\le2\}.
\]
Such a step either has $E_k\ge\eps$ or has $E_k<\eps$ and $g_k=0$.
By \eqref{eq:energybudget}, there are at most $L/\eps$ steps of the
first type. For the second type, Lemma~\ref{lem:common} implies
$|\alpha_k|\le M_0$.

Every orbit in $\alpha_k$ has action at most $L$. Thus only $N(L)$ simple
orbits can occur. The ECH generators $\alpha_k$ are distinct because
their actions are strictly increasing. Therefore
\begin{equation}\label{eq:D}
\begin{split}
 D&\le\frac L\eps+\sum_{q=0}^{M_0}\binom {N(L)} q\\
  &\le\frac L\eps+C_0(1+N(L))^{M_0}=O(L+(\log L)^{M_{0}})=O(L).
\end{split}
\end{equation}

\medskip
\noindent\textbf{Step 4. The total cost of the exceptional steps.}
For a real number $x$, write $x_+=\max(x,0)$. We have $j_k-2=0$ on
$\Gcal$. Thus every negative contribution to \eqref{eq:jbudget} comes
from one of the $D$ steps above. By \eqref{eq:stepj}, each such
contribution is at least $-3$. Hence
\begin{equation}\label{eq:jpositive}
 \sum_{k=1}^K(j_k-2)_+\le C_*L+3D.
\end{equation}
Define
\begin{equation}\label{eq:Wdef}
 W=\sum_{k\in\Bcal}(j_k+2).
\end{equation}
If $j_k\ge3$, then $j_k+2\le5(j_k-2)$. If $j_k\le2$, then
$j_k+2\le4$. Consequently,
\begin{equation}\label{eq:Wbound}
\begin{split}
 W&\le5\sum_{k=1}^K(j_k-2)_++4D\\
  &\le5C_*L+19D=O(L).
\end{split}
\end{equation}
Since $j_k+2\ge1$, we also have
\begin{equation}\label{eq:Bcount}
 |\Bcal|\le W.
\end{equation}
Thus the sum of $j_k+2$ over the exceptional steps is $O(L)$. We use this
sum below to bound the number of orbits which these steps can change.

\medskip
\noindent\textbf{Step 5. The local weights bound the remaining steps.}
We use the function $h$ from Lemma~\ref{lem:weights}. It is fixed on the
entire set of simple periodic orbits. Put
\begin{equation}\label{eq:potential}
 \Phi(\alpha)=\sum_{\gamma\in\alpha}h(\gamma).
\end{equation}
For the generators considered here,
\begin{equation}\label{eq:phirange}
 0\le\Phi(\alpha_k)\le B \cdot N(L).
\end{equation}
For $k\in\Gcal$, equations \eqref{eq:exchange} and
\eqref{eq:weightorder} imply
\begin{equation}\label{eq:goodincrement}
 \Phi(\alpha_k)-\Phi(\alpha_{k-1})\ge1.
\end{equation}
For an arbitrary step, every orbit in the symmetric difference
$\alpha_k\triangle\alpha_{k-1}=(\alpha_{k}\backslash \alpha_{k-1 })\cup (\alpha_{k-1}\backslash \alpha_{k}) $ occurs as an end of $C_k$. Hence
\[
 |\alpha_k\triangle\alpha_{k-1}|\le e_k
 =j_k+2-2g_k\le j_k+2.
\]
The common orbits cancel in the difference of the two sums in
\eqref{eq:potential}. Since $h\le B$, we have
\begin{equation}\label{eq:allincrement}
 |\Phi(\alpha_k)-\Phi(\alpha_{k-1})|
 \le B(j_k+2).
\end{equation}

Since
$\{1,\ldots,K\}=\Gcal\sqcup\Bcal$,  
\begin{align*}
 |\Gcal|
 &=\sum_{k\in\Gcal}1
 \le\sum_{k\in\Gcal}(\Phi(\alpha_k)-\Phi(\alpha_{k-1}))\\
 &=\sum_{k=1}^{K}(\Phi(\alpha_k)-\Phi(\alpha_{k-1}))-\sum_{k\in\Bcal}(\Phi(\alpha_k)-\Phi(\alpha_{k-1}))\\
 &=\Phi(\alpha_K)-\Phi(\alpha_0)
    -\sum_{k\in\Bcal}\bigl(\Phi(\alpha_k)-\Phi(\alpha_{k-1})\bigr)\\
 &\le B\cdot N(L)+\sum_{k\in\Bcal}|\Phi(\alpha_k)-\Phi(\alpha_{k-1})|\\
 &\le B \cdot N(L)+B\sum_{k\in\Bcal}(j_k+2)\\
 &=B  \cdot N(L)+BW.
\end{align*}
Here we use
$\Phi(\alpha_K)-\Phi(\alpha_0)\le B N(L)$
and $-a\le |a|$.

This estimate and \eqref{eq:Bcount} give
\[
 K=|\Gcal|+|\Bcal|\le B \cdot N(L)+(B+1)W.
\]
Now \eqref{eq:logarithmic} and \eqref{eq:Wbound} give $K=O(L)$.
This proves Proposition~\ref{prop:linear}.
\end{proof}

\section{Proof of the main theorem}\label{sec:mainproof}

\begin{proof}[\bf Proof of Theorem~\ref{thm:negative}]
Suppose that \NH  holds. Choose a fixed $U$-sequence
$\{\alpha_k\}_{k\ge0}$ in sufficiently large degrees in the class
$\Gamma=0$, as in \eqref{eq:tower}. Put $L_K=\A(\alpha_K)$. By
Proposition~\ref{prop:linear},
\[
 \frac{K}{L_K}\le C
\]
for all sufficiently large $K$. On the other hand, the Weyl law
\eqref{eq:towerasymp} gives
\[
 \frac{L_K^2}{2K}\longrightarrow \Vol(Y,\lambda) >0.
\]
Thus $L_K\to\infty$ and
\[
 \frac{K}{L_K}
 =\frac{L_K}{L_K^2/K}\longrightarrow\infty.
\]
This is a contradiction. Therefore  the simple periodic orbits cannot
all be negative hyperbolic.
\end{proof}

\begin{proof}[\bf Proof of Theorem~\ref{thm:main}]
If $b_1(Y)>0$, the conclusion follows from \cite[Proposition 1.9]{CHP}.
Suppose that $b_1(Y)=0$. Then $c_1(\xi)$ is torsion.
By \cite[Theorem 1.4]{CHP}, there are either two or infinitely many
simple periodic orbits.
Since there are at least three, there are infinitely many.

If an elliptic orbit exists, \cite[Theorem 1.6]{S} gives a simple
positive hyperbolic orbit. Suppose that there is no elliptic orbit.
If there were no simple positive hyperbolic orbit either, all simple
orbits would be negative hyperbolic. This contradicts
Theorem~\ref{thm:negative}. This completes the proof.
\end{proof}

\end{document}